\documentclass[letterpaper,11pt]{amsart}

\usepackage[margin=1.2in]{geometry}
\usepackage{amsmath,amsthm,amssymb}
\usepackage{xspace,xcolor}
\usepackage[
  breaklinks,colorlinks,
  citecolor=teal,linkcolor=teal,urlcolor=teal
]{hyperref}
\usepackage[alphabetic]{amsrefs}
\usepackage[all]{xy}
\usepackage{booktabs}

\theoremstyle{plain}
\newtheorem{thm}{Theorem}[section]

\newtheorem{lem}[thm]{Lemma}
\newtheorem{prop}[thm]{Proposition}
\newtheorem{cor}[thm]{Corollary}

\theoremstyle{definition}

\newtheorem{eg}[thm]{Example}

\newtheorem{question}[thm]{Question}

\theoremstyle{remark}
\newtheorem{rmk}[thm]{Remark}

\numberwithin{equation}{section}

\def\N{{\mathbb N}}

\def\C{{\mathbb C}}

\def\P{{\mathbb P}}

\def\cL{\mathcal{L}}

\def\cN{\mathcal{N}}

\def\I{\mathcal{I}}

\def\O{\mathcal{O}}

\def\f{\phi}
\def\ff{\psi}
\def\e{\eta}

\def\n{\nu}

\def\Om{\Omega}

\def\.{\cdot}
\def\^{\widehat}
\def\~{\widetilde}
\def\o{\circ}
\def\ov{\overline}

\def\rat{\dashrightarrow}

\def\({\left(}
\def\){\right)}

\renewcommand{\and}{ \ \ \text{ and } \ \ }

\def\reg{\mathrm{reg}}

\DeclareMathOperator{\Bl} {Bl}

\DeclareMathOperator{\Hilb} {Hilb}

\DeclareMathOperator{\Rat} {Rat}
\DeclareMathOperator{\SRC} {SRC}

\begin{document}



\title     {Rationality in families of threefolds}


\author    {Tommaso de Fernex}
\author    {Davide Fusi}
\address   {Department of Mathematics,
            University of Utah,
            155 South 1400 East,
            Salt Lake City, UT 84112, USA} 
\email     {defernex@math.utah.edu}
\email     {fusi@math.utah.edu}



\thanks    {The first author is partially supported by NSF CAREER Grant DMS-0847059.}

\thanks    {Compiled on \today. Filename \small\tt\jobname}


\begin{abstract}
We prove that in a family of projective threefolds defined over
an algebraically closed field,
the locus of rational fibers is a countable union of closed subsets
of the locus of separably rationally connected fibers.
When the ground field has characteristic zero, this implies that
the locus of rational fibers in a smooth family of projective threefolds
is the union of at most countably many closed subfamilies.
\end{abstract}

\maketitle

\section{Introduction}

Let $f\colon X \to T$ be a projective equidimensional 
morphism onto a connected reduced scheme $T$ of finite type
over an algebraically closed field $k$. Assume that the fibers $X_t := f^{-1}(t)$
are varieties for all $t \in T$, and
let $n$ denote the relative dimension of $f$.
We will refer to $f$ as a family of projective varieties. 
We are interested in understanding the algebraic 
structure of the \emph{rational locus}
\[
\Rat(f) := \{\, t \in T \mid \text{$X_t$ is a rational variety}\,\}
\]
of the family.

It follows by general facts that 
$\Rat(f)$ is a countable union of locally closed subsets of $T$. 
Once singularities are allowed, it is easy to pick up 
examples of families of rational varieties that specialize to nonrational ones. 
In characteristic zero, however, the following question 
regarding smooth families has been around for some time.

\begin{question}
\label{conj}
Assuming that $f \colon X \to T$ is a smooth family of projective varieties
over an algebraically closed field of characteristic zero, 
is $\Rat(f)$ equal to a countable union of closed subsets of $T$?
\end{question}

The answer is easy and well-known in dimension one, and follows from
Castelnuovo's rationality criterion in dimension two. 
It is expected that in higher dimensions $\Rat(f)$ can be a proper subset, 
possibly with infinitely many components; 
this should occur for instance in smooth families of cubic fourfolds.

In this paper we study the three-dimensional case. 
We do not put conditions on the characteristic of the ground field $k$;
for this reason we consider the \emph{separably rationally connected locus}
\[
\SRC(f) := \{\, t \in T \mid \text{$X_t$ is separably rationally connected}\,\}
\]
of the family.
There is an inclusion $\Rat(f) \subset \SRC(f)$, and equality holds for 
projective families of relative dimension $n \le 2$. We prove the following result.

\begin{thm}
\label{t:main}
For every family $f\colon X \to T$ of projective 
varieties of dimension three over an algebraically closed field,
$\Rat(f)$ is a countable union of closed subsets of $\SRC(f)$. 
\end{thm}

In the case where the ground field has characteristic zero, 
this yields a positive answer to Question~\ref{conj} in dimension three.

\begin{cor}
\label{c:conj3}
For a smooth family $f \colon X \to T$ of projective threefolds 
over an algebraically closed field of characteristic zero, 
$\Rat(f)$ is a countable union of closed subsets of $T$.
\end{cor}

The proof of these results is based on two basic properties: the countability
and properness of the irreducible components of Hilbert schemes, and 
the fact that divisorial valuations are geometric. In characteristic zero, 
one can alternatively use the Weak Factorization Theorem in place
of the property on valuations. A key step of the proof is a
result regarding one-parameter degenerations of rational varieties, 
which is stated and proven below in Theorem~\ref{t:closed}.
A special case of this result, where the ground field is $k = \C$
and the degeneration is given by a smooth family of complex threefolds, 
was also obtained using Hodge theoretic methods by Timmerscheidt \cite{Tim}.

In view of Theorem~\ref{t:main}, one could ask, as a 
plausible extension of Question~\ref{conj} to arbitrary 
characteristics, whether given any family $f\colon X \to T$ of projective 
varieties over an algebraically closed field, 
$\Rat(f)$ is always a countable union of closed subsets of $\SRC(f)$.
The examples discussed below in Example~\ref{ex:not-closed} 
suggest however that this may be false in higher dimensions, even in characteristic zero;
in fact, in view of these examples, it seems likely
that the hypothesis on the dimension in the theorem
is optimal. 
It is important to remark that the families 
in these examples are not smooth, so they do not bring enough
evidence to disbelieve Question~\ref{conj}. Rather, they suggest that if 
a positive answer is expected to this question, then the smoothness of the family should
play a key role in the proof. 

We do not know the full history of Question~\ref{conj}.
It is an old problem to find an example of a family of nonrational projective varieties
that specializes to a smooth rational variety.
The question whether in smooth families 
the locus of rational fibers can always be expressed as a countable union of closed
subsets has been considered in conversations between Paolo Francia
and Alessandro Verra; it is likely that the same question has been 
raised by other mathematicians as well. 
The special case of cubic hypersurfaces in $\P^5$ is representative:
rationality questions about cubic fourfolds have attracted 
the attention of the mathematical community for a long time,
starting with the work of Ugo Morin \cite{Mor40} and Gino Fano \cite{Fan44} if not earlier.
The construction of countably many families of rational cubic fourfolds
due to Brendan Hassett \cite{Has99,Has00}, in particular, fits naturally in the context  of 
Question~\ref{conj} and has prompted more people to consider 
the rationality problem from this point of view.

\subsection*{Acknowledgment}
The first author would like to express his gratitude to Paolo Francia
who first got him interested in Question~\ref{conj};
we dedicate this paper to his memory.
We would like to thank Emanuele Macr\`i, Massimiliano Mella,  
and Alessandro Verra for valuable comments, and Claire Voisin for
explaining to us the argument of the proof of Proposition~\ref{p:loc-closed}
given below. We thank J\'anos Koll\'ar from bringing
the paper \cite{Tim} to our attention, and the referees for useful comments.

\section{General properties}

We work over an algebraically closed field $k$.
All schemes are assumed to be of finite type over $k$.
With the term \emph{variety} we mean an integral scheme.
A morphism of varieties $f \colon X \to Y$ is \emph{separable} 
if it is dominant and the field extension $K(X) \supset K(Y)$ is separably generated.
A variety $X$ is \emph{rational} if its function field is purely transcendental over $k$, 
or, equivalently, if $X$ is birational to $\P^n$
where $n = \dim X$. A variety $X$ is \emph{separably rationally connected}
if there is a variety $V$ and a morphism $u \colon \P^1 \times V \to X$
such that 
\[
u^{(2)} \colon \P^1 \times \P^1 \times V \to X \times X
\]
is separable, or equivalently, is dominant and smooth at the generic point 
(cf. \cite[Definition~IV.3.2]{Kol96}).

The following property is a 
direct consequence of the definition (cf.~\cite[Proposition~IV.3.3.1]{Kol96}).

\begin{prop}
\label{p:SRC-bir}
If $X$ and $X'$ are two proper varieties that are birationally equivalent, 
then $X$ is separably rationally connected if and only if $X'$ is.
\end{prop}

It is straightforward from the definitions that a proper rational variety
is separably rationally connected, and the converse holds in dimension two
(cf.~\cite[Exercise~IV.3.3.5]{Kol96}).

\begin{prop}
\label{p:SRC-dim2}
Let $X$ be a proper surface. 
If $X$ is separably rationally connected, then it is rational.
\end{prop}

\begin{proof}
By \cite{Lip78}, there exists a resolution of singularities of $X$. 
Since both rationality and separably rational connectedness are
birational properties, we may thus assume without loss of generality that $X$ is smooth. 
Then, by \cite[Theorem~IV.3.7]{Kol96}, there is a morphism $g \colon \P^1 \to X$ such that
$f^*T_X$ is ample. This implies that every section of $(\wedge^q\Om_X)^{\otimes m}$, 
for any $q,m \ge 1$, vanishes along $g(\P^1)$. As these curves cover a dense set in $X$, 
we conclude that all sections of $(\wedge^q\Om_X)^{\otimes m}$ are zero. 
Therefore $X$ is rational by Castelnuovo's criterion.
\end{proof}

Suppose now that $f\colon X \to T$
is a family of projective varieties parameterized by a connected, reduced scheme $T$
of finite type over $k$.
The rational locus $\Rat(f)$ of the family has the following algebraic structure. 

\begin{prop}
\label{p:loc-closed}
$\Rat(f)$ is a countable union of locally closed subsets of $T$.
\end{prop}

We learned the following proof, which simplifies our original arguments, 
from Claire Voisin.

\begin{proof} 
Let $P := T \times \P^n$, where $n$ is the relative dimension of $f$.
First observe that every closed subscheme $Z \subset X \times_T P$
determines a birational map $X_t \rat P_t \cong \P^n$ for every $t$
such that $Z_t$ is irreducible and 
both projections $Z_t \to X_t$ and $Z_t \to P_t$ are birational;
conversely, all birational maps from fibers of $f$ to $\P^n$
arise in this way. 

Consider the relative Hilbert scheme $H := \Hilb(X\times_T P/T)$
of $X \times_T P$, and let $U \to H$ be the universal family:
$U$ is a closed subscheme of $X \times_T P \times_T H$, flat over $H$.
For every irreducible component $H_j$ of $H$, consider the set
of points $h \in H_j$ such that $U_h$ is irreducible and, if $t \in T$ is the image
of $h$, then the projections
$U_h \to X_t$ and $U_h \to P_t$ are birational. By applying Lemma~\ref{l:constr}
to $U_j \to X \times_T H_j \to H_j$ and $U_j \to P \times_T H_j \to H_j$
where $U_j := U \times_H H_j$, we see that
this set is constructible in $H_j$. By Chevalley's theorem, its image in $T$ is 
also constructible, and as such
can be written as a finite union of locally closed subsets.
The union of all these sets, as $H_j$ varies among the irreducible components
of the Hilbert scheme, is $\Rat(f)$. The statement then follows by the fact that
the Hilbert scheme has countably many irreducible components. 
\end{proof}

\begin{lem}
\label{l:constr}
Let $U \to V \to H$ be morphisms of schemes of finite type over $k$, with $U \to H$ flat and 
$V \to H$ projective. Then the set $h \in H$ such that
$U_h$ is irreducible and $U_h \to V_h$ is birational is constructibe. 
\end{lem}

\begin{proof}
Assuming without loss of generality that $H$ is irreducible,
these properties hold at the generic point of $H$ if and only if they hold
over a nonempty open set of $H$. The statement then follows by Noetherian induction.
\end{proof}

\begin{rmk}
An analogous property is satisfied by the locus of unirational 
varieties: the argument easily adjusts to this case by relaxing the 
condition on $U_h \to X_t$ from being birational 
to being dominant. A related result concerning the behavior of uniruledness in families 
is proven in \cite[Therem~IV.1.8]{Kol96}, where it is shown that the locus of 
uniruled varieties in an equidimentional proper family is a countable union of 
closed subsets of the base. 
\end{rmk}

Regarding the general structure of $\SRC(f)$, 
several interesting cases are covered by the following proposition.

\begin{prop}
\label{p:SRC}
Let $f \colon X \to T$ as above. 
\begin{enumerate}
\item
In any setting where embedded resolution of singularities exists, 
$\SRC(f)$ is a constructible subset of $T$. 
\item
If $f$ is smooth, then $\SRC(f)$ is open in $T$.
\item
If $f$ is smooth and $k$ has characteristic zero, then 
$\SRC(f)$ is open and closed in $T$ (and thus is either empty or equal to $T$).
\end{enumerate}
\end{prop}

\begin{proof}
The assertions in~(b) and~(c) are proven in \cite[Theorem~IV.3.11]{Kol96}.
Regarding~(a), first note that $f$ is separable as it has reduced fibers
(cf.~\cite[Theorem~II.8.6A and Proposition~II.8.10]{Har77}) , 
and so is the restriction of $f$ over any locally closed subset of $T$. 
Let $Y \to X$ be a resolution of singularities, and 
consider the composition map $g \colon Y \to T$. 
Since $g$ is separable, there is a non-empty open set $T^\o$ in the regular locus of $T$ 
over which the induced map $g^\o \colon g^{-1}(T^\o) \to T^\o$ is smooth
(the proof of \cite[Corollary~III.10.7]{Har77} goes through without assumptions on the 
characteristic of the ground field as long as one assumes that the morphism is separable).
By~(b), $\SRC(g^\o)$ is an open subset of $T^\o$.
Note on the other hand that $\SRC(f) \cap T^\o = \SRC(g^\o)$ by Proposition~\ref{p:SRC-bir}, 
since every fiber of $g^\o$ is birational to the corresponding 
fiber of $f$.
Thus the assertion follows by Noetherian property, by considering a suitable 
stratification of $T$. 
\end{proof}

\begin{rmk}
A different definition of separably rational connectedness
has been considered in works of de Jong, Graber, and Starr
(cf.~\cite{dJS03,Gra06}), where a projective variety $X$ is
said to be separably rationally connected
if there exists a morphism $g \colon \P^1 \to X_\reg$
such that $g^*T_{X_\reg}$ is ample. 
It is elementary to see that this property is open in families.
It follows by the deformation theory of rational curves 
(see for instance the argument in the proof of \cite[Theorem~IV.3.5]{Kol96})
that a projective variety that is 
separably rationally connected in the sense of de Jong, Graber and Starr 
is also separably rationally connected in the sense defined in the previous section,
and the two notions coincide whenever $X$ is smooth
by \cite[Theorem~IV.3.7]{Kol96}. It is however not clear to us whether 
being separably rationally connected in the sense of de Jong, Graber and Starr
is a birational property among projective varieties. 
In particular, we do not know whether a rational
projective variety is necessarily separably rationally connected in this sense.
\end{rmk}

It is easy to construct examples of families of rational projective varieties degenerating
to singular varieties that are not rational, and vice versa. 
We do not know any example of a (connected) smooth
family of projective varieties containing both rational and nonrational members. 
It is expected in general that one needs to consider countable unions 
in Proposition~\ref{p:loc-closed} and in Question~\ref{conj}. 

\begin{eg}
Complex cubic fourfolds in $\P^5$ form a particularly interesting
class of varieties from the point of view of rationality.  
The quest for rational examples goes back at least to Morin \cite{Mor40}, 
who gave an incorrect argument that would have implied that the general cubic
in $X \subset \P^5$ is rational. 
In the same paper, however, Morin correctly proves the rationality of
general Pfaffian cubic fourfolds: these 
span a codimension one family of smooth rational cubics fourfolds
which was further studied by Fano \cite{Fan44}, 
Tregub \cite{Tre84}, and  Beauville and Donagi \cite{BD85}.
A crucial step in the study of cubic fourfolds is Voisin's proof of
a Torelli Theorem for these varieties \cite{Voi86}. 
More examples of rational cubic fourfolds were found by Zarhin \cite{Zar90}, and
later Hassett \cite{Has99,Has00} constructed a countable series of distinct families
of smooth rational cubic fourfolds: these
are parameterized by divisors on the family of cubics containing a plane, 
which has codimension one in the whole space of cubics. 
It is expected on the other hand that not only the general cubic in $\P^5$, but also
the very general element among those containing a plane is not rational.
An explicit conjecture has been formalized in the language of 
derived categories by Kuznetsov \cite{Kuz10}. Knowing this conjecture would give
an example of a family where the rational locus is, strictly speaking, a countable union
of closed subfamilies. 
\end{eg}

\section{The three dimensional case}

In dimension three, we have the following 
property regarding one-parameter degenerations of rational projective varieties.

\begin{thm}
\label{t:closed}
Let $f\colon X \to T$ be a projective morphism from a variety $X$ onto
a smooth curve $T$ defined over an uncountable algebraically closed field $k$. 
Let $0 \in T$ be a closed point.
Assume that $X_t$ is a rational variety for every $t \ne 0$.
Then every reduced, irreducible component $D$ of $X_0$ that is separably rationally connected
is rational.
\end{thm}

\begin{proof}
Using the same notation as in the proof of Proposition~\ref{p:loc-closed}, 
there are countably many irreducible locally closed subsets
\[
S_i \subset H := \Hilb(X \times_T (T \times \P^3)/T), \quad i \in \N,
\] 
such that, if $U$ is the universal family of $H$ then,  
for every $h \in S_i$, $U_h$ gives a birational correspondence between a fiber
$X_t$ of $X$ and $\P^3$, and the rational locus of $f$ is given by the union
\[
\Rat(f) = \bigcup_i T_i
\]
of the images $T_i \subset T$ of the sets $S_i$.
Note that each $T_i$ is an irreducible constructible subset of $T$, and thus
is either a point or an open subset. 

Since we are assuming that $\Rat(f) = T \smallsetminus \{0\}$ and the ground
field is uncountable, there is at least one index $i_0$ such that $T_{i_0}$ is
a dense open subset of $T$.
Let $\ov S_{i_0}$ be the closure of $S_{i_0}$ in $H$.
By the properness of the Hilbert scheme over $T$, $\ov S_{i_0}$ maps onto $T$. 
Let $C \subset \ov S_{i_0}$ be a general complete intersection curve; 
we assume in particular that $C$ is irreducible, that it intersects $S_{i_0}$, 
and that the map $C \to T$ is surjective. Note that $C \cap S_{i_0}$ is open and dense in $C$. 
Let then $T' \to C$ be the normalization, let 
$g \colon T' \to T$ be the composition map, and fix a point $0' \in g^{-1}(0)$. 
The fiber product $X' := T' \times_T X$ is a variety, and
by base change we obtain a projective morphism 
\[
f' \colon X' \to T'
\]
with fiber $X'_{0'} \cong X_0$ over $0'$.

Let $H' := \Hilb(X' \times_{T'} (T' \times \P^3)/T')$, with universal family $U'$. 
By base change, we have a commutative diagram
\[
\xymatrix{
U_{T'} \ar[r] \ar[d] & U' \ar[d] \ar[r] & U \ar[d] \\
T' \ar[r] \ar@{=}[rd] & H' \ar[d]\ar[r] & H \ar[d] \\
& T' \ar[r] & T
},
\]
where
\[
U_{T'} := U \times_H T' = U' \times_{H'} T'
\]
is the pullback of the universal family to $T'$. 
By construction, the image of the induced map
$U_{T'} \to X' \times_{T'} (T' \times \P^3)$ is a closed subscheme 
\[
Z' \subset X' \times_{T'} (T' \times \P^3)
\] 
such that for every $s \in g^{-1}(T_{i_0})$ the fiber $Z'_s$ is irreducible and 
both projections $Z'_s \to X'_s$ and $Z'_s \to \{s\} \times \P^3$ are birational. 
Since $g^{-1}(T_{i_0})$ is an open dense subset of $T'$, it follows that
the support of $Z'$ is the graph of a birational map 
\[
\f \colon X' \rat T' \times \P^3
\]
defined over $T'$. 

Let $D'$ be the irreducible component of $X'_{0'}$ 
mapping to $D$ via the isomorphism $X'_{0'} \cong X_0$. 
Since the fiber $X'_{0'}$ is a Cartier divisor on $X'$
that is reduced at the generic point $\e_{D'}$ of $D'$, 
$\e_{D'}$ is contained in the regular locus of $X'$.
Thus the vanishing order at $\e_{D'}$ defines 
a divisorial valuation on the function field of $X'$. 
Let $\n$ be the induced valuation on the function field
of $T' \times \P^3$. Note that the center $C_0$ of $\n$ in $T' \times \P^3$ is contained in the 
fiber $\{0'\} \times \P^3$. 

Consider the sequence of blow-ups
\[
\cdots\quad \to Y_i \to Y_{i-1} \to \quad\cdots\quad \to Y_1 \to Y_0 := T' \times \P^3
\]
where each $g_i \colon Y_i \to Y_{i-1}$ is the blow-up
of $Y_{i-1}$ along the the center $C_{i-1}$ of $\n$.
Note that, for every $i$, $C_i$ is contained in the exceptional 
divisor of the blow-up $g_i$, and $g_i(C_i) = C_{i-1}$.

By induction on $i$, both $Y_{i-1}$ and $C_{i-1}$ are smooth at the 
generic point of $C_{i-1}$, and therefore  
there is a dense open set $Y^\o_{i-1} \subset Y_{i-1}$, 
contained in the regular locus of $Y_{i-1}$, 
such that $C_{i-1}^\o := C_{i-1} \cap Y^\o_{i-1}$ is smooth and the induced map
$g_i^{-1}(Y^\o_{i-1}) \to Y^\o_{i-1}$
is the blow-up of the normal bundle $\cN_{i-1}$ of $C^\o_{i-1}$ in $Y^\o_{i-1}$. 
In particular, the restriction of the exceptional locus
of $g_i$ over $C^\o_{i-1}$ is isomorphic to the projective 
bundle $\P_{C^\o_{i-1}}(\cN_{i-1})$.

It follows by a theorem of Zariski (cf.~\cite[Lemma~2.45]{KM98}) that
there is an integer $m \ge 0$ such that the 
center $C_m$ of $\n$ has codimension one in $Y_m$ and $\n$ is given by
the order of vanishing at the generic point of $C_m$. 
In particular, $C_m$ is birational to $D'$ since both their function fields
are equal to the residue field of the valuation 
(geometrically, $C_m$ is the proper transform of $D'$ under 
the birational map $X' \rat Y_m$).
We can pick $m$ to be the least integer with these properties. 

If $m = 0$, then the center of $\n$ in $T' \times \P^3$ is the 
whole fiber $\{0'\} \times \P^3$. This means that $\f$
induces a birational map from $D'$ to $\{0'\} \times \P^3$, 
and therefore $D'$ is rational. 

Suppose then that $m \ge 1$. In this case the projection $C_m \to C_{m-1}$ is 
a surjective map from a threefold to a variety of dimension at most two. 
Note that $C_m$ is separably rationally connected, since it
is birational to $X_0$ which is separably rationally connected by hypothesis, 
and being separably rationally connected is a birational property
(see Proposition~\ref{p:SRC-bir}).
Since the map $C_m \to C_{m-1}$ is smooth over $C^\o_{m-1}$, it follows 
that $C_{m-1}$ is separably rationally connected too. 
The assumption on the relative dimension of $f$ implies that
$\dim C_{m-1} \le 2$. If $C_{m-1}$ has dimension at most one then 
it is clearly rational, and the same conclusion holds if $C_{m-1}$ is a surface
by Proposition~\ref{p:SRC-dim2}. 
Note, on the other hand, that 
$C_m$ contains $g^{-1}(C^\o_{m-1})$ as a dense open set, and the latter
is isomorphic to $\P_{C^\o_{m-1}}(\cN_{m-1})$. 
We conclude that $C_m$ is rational. Therefore $D$ is rational.
\end{proof}

\begin{rmk}
\label{r:WFT}
If the ground field $k$ has characteristic zero then 
one can use an alternative argument, based on the
Weak Factorization Theorem
\cite{AKMW02,Wlo03}, to prove Theorem~\ref{t:closed}.
The argument goes as follows. 
Let $\f\colon X' \rat T' \times \P^3$ and $D'$ be as in the proof of the theorem, 
and suppose that $\f$ contracts $D'$
(so that it does not induce directly a birational map from $D'$ to $\{0'\} \times \P^3$).
Let $Y \to X'$ be a resolution of singularities. 
By the Weak Factorization Theorem applied to the
induced birational map $Y \rat T' \times \P^3$, we can find a sequence
of blow-ups $p_i$ and blow-downs $q_j$ with smooth irreducible centers
\[
\xymatrix@C=18pt@R=18pt{
& Z^1 \ar[dl]_{p_1}\ar[dr]^{q_1} && Z^2  \ar[dl]_{p_2}\ar[dr]^{q_2} 
&&&& Z^n \ar[dl]_{p_n}\ar[dr]^{q_n} & \\
Y=Y^0 && Y^1 && Y^2 & \dots & Y^{n-1} && Y^n = T'\times\P^3
}
\]
(we allow isomorphisms among the maps $p_i$ and $q_j$).
Since $\f$ contracts $D'$, there is a model $Z^i$,
for some $1 \le i \le n$, where the proper transform $D^i$ of $D'$
is the exceptional divisor of $q_i \colon Z^i \to Y^i$. 
Since $D^i$ is rationally connected, so is its image $W_i := q_i(D^i)$, 
which is therefore rational.
This implies that $D^i$ is rational, as it is isomorphic to the
projectivization of the normal bundle of $W_i$ in $Y^i$.
Therefore $D$ is rational.
\end{rmk}

\begin{rmk}
\label{r:Tim}
When the ground field is $k = \C$ and the family $f\colon X \to T$ is smooth, 
Theorem~\ref{t:closed} also follows by \cite[Theorem~1]{Tim}.
\end{rmk}

\begin{proof}[Proof of Theorem~\ref{t:main}]
The statement of the theorem is trivial if the ground field $k$ if finite or countable, 
since in this case any subset of $T$
can be expressed as a countable union of closed subsets. 
Thus we can assume that $k$ is uncountable. 

By Proposition~\ref{p:loc-closed}, $\Rat(f)$ is a countable 
union of locally closed subsets of $R_i \subset T$.
Suppose that $\Rat(f)$ cannot be written as a countable 
union of closed subsets of $\SRC(f)$. 
Then we can find a point $p \in \SRC(f) \smallsetminus \Rat(f)$
that belongs to the closure $\ov R_i$ of $R_i$ in $T$ for some $i$. Let
$S \subset \ov R_i$ be a curve passing through $p$ 
and with generic point in $R_i$.
Let $\~S \to S$ be the normalization of $S$ and fix a point $0 \in \~S$
in the pre-image of $p$. Let then $\~T \subset \~S$ be an open neighborhood of $0$
such that $\~T \smallsetminus \{0\}$ maps into $R_i$. 
By taking the base change 
\[
\~f \colon \~X := X \times_T \~T \to \~T,
\]
we reduce to the setting of Theorem~\ref{t:closed}, 
which implies that $\~X_0$ is rational. Since $\~X_0 \cong X_p$, 
this contradicts the fact that $p \not\in\Rat(f)$.
\end{proof}

\begin{proof}[Proof of Corollary~\ref{c:conj3}]
In the hypothesis of Question~\ref{conj}, 
assume that $f$ has relative dimension 3. 
Suppose that $\Rat(f) \ne \emptyset$. Then $\SRC(f)$ is non-empty, 
and thus it is equal to $T$ by Proposition~\ref{p:SRC}.
Therefore the corollary reduces to a special case of Theorem~\ref{t:closed}.
\end{proof}

\begin{eg}
\label{ex:not-closed}
Consider the projection $\P^1 \times \P^n \to \P^1$.
Fix a point $0 \in \P^1$, and let $W$ be a smooth hypersurface
of degree $n$ in the fiber $\{0\} \times \P^n$. 
By \cite{dF}, $W$ is nonrational if $n \ge 4$ and the ground field has characteristic zero. 
Although $W$ might be {\it stably rational} (which would mean that $W \times \P^m$
is rational for some $m$), it is quite possible that $W \times \P^1$ is 
nonrational. In fact, it is conceivable (and possibly expected) that $W$ is
not even unirational if it is general, 
and this would certainly imply that $W \times \P^1$ is not rational. 
Now, let $\cL = \O_{\P^1}(2) \boxtimes \O_{\P^n}(n)$ and 
$\I_W$ be the ideal sheaf of $W$ in $\P^1 \times \P^n$. The sheaf
$\cL \otimes \I_W$ is globally generated, and thus
the linear system $|\cL \otimes \I_W|$ defines a rational map 
\[
\ff \colon \P^1 \times \P^n \rat X \subset \P H^0(\cL \otimes \I_W)
\]
which is resolved by the blow-up $Y = \Bl_W(\P^1 \times \P^n)$ of $\I_W$. 
Here $X$ denotes the closure of the image of the map.
The map $\ff$ is defined over $\P^1$, and thus
there is a morphism $f \colon X \to \P^1$. 
Furthermore, $\ff$ induced an isomorphism away from the fibers over $0$, 
so that $X_t \cong \P^n$ for $t \ne 0$. 
On the other hand the fiber $X_0$ is birational to $W \times \P^1$. 
Indeed, the induced morphism $\ff' \colon Y \to X$ contracts the proper transform 
of $\{0\} \times \P^n$ to a point and maps the exceptional
divisor of the blow-up birationally to the fiber $X_0$, which is thus
isomorphic to the projective cone in $\P^{n+1}$ over $W$. 
In particular, $X_0$ is not rational if $W \times \P^1$
is not rational; note however that $X_0$ is always separably rationally connected
(for every $W$ in characteristic zero, and for general $W$ in positive characteristics
by \cite{Zhu}).
This example suggests that the analogous statement of Theorem~\ref{t:closed} 
in higher dimensions may be false, possibly starting from dimension four.
\end{eg}



\setlength{\parindent}{0in}
\def\scshape{}

\end{document}